\documentclass{article}     

\usepackage{amssymb,amsmath,mathrsfs, amsfonts, graphicx,theorem,dsfont,yfonts, mathabx, version, enumitem}
\usepackage{pstricks}
\usepackage{pst-all}

\usepackage[normalem]{ulem}

\usepackage[english]{babel}
\usepackage{cancel}

\usepackage[left,modulo]{lineno}
\def\mN{{\mathbb N}}

\def\cE{\mathscr E}

\def\cK{\mathcal K}

\def\cL{\mathcal L}

\def\a{\alpha}

\def\f{\varphi}

\def\cl{\texttt{cl}}

\newtheorem{theorem}{Theorem}
\newtheorem{lemma}[theorem]{Lemma}
\newtheorem{definition}[theorem]{Definition}
\newtheorem{prop}[theorem]{Proposition}

\newtheorem{remark}[theorem]{Remark}

\newtheorem{proof}[theorem]{Proof}

\newtheorem{cor}[theorem]{Corollary}
\newtheorem{conj}{Conjecture}

\def\og{\leavevmode\raise.3ex\hbox{$\scriptscriptstyle\langle\!\langle$~}}
\def\fg{\leavevmode\raise.3ex\hbox{~$\!\scriptscriptstyle\,\rangle\!\rangle$}}

\newsavebox{\fmbox}

\def\rd#1{\:\searrow \!\!\stackrel{#1}{}\:}
\def\ep#1{\:\stackrel {#1}{}\!\!\nearrow\:}

\def\cubion{\mathfrak Q}

\def\smallcubion{\mathfrak Q}

\def\parasol{\mathfrak P}

\def\dom{\vdash}

\makeatletter
\def\timenow{\@tempcnta\time
  \@tempcntb\@tempcnta
  \divide\@tempcntb60
  \ifnum10>\@tempcntb0\fi\number\@tempcntb
  \multiply\@tempcntb60
  \advance\@tempcnta-\@tempcntb
  :\ifnum10>\@tempcnta0\fi\number\@tempcnta}
\makeatother
\usepackage{calc}
\newcounter{hours}\newcounter{minutes}

\newcommand{\N}{\mathbb{N}}\vspace{1 mm}

\begin{document}
{\large
\bigskip
\vspace{7 mm}

\centerline{\large \bf A hierarchy of dismantlings in graphs}
%\vspace{1 mm}

%\centerline{version  \today,  \printtime}
\vspace{5 mm}

\centerline{Etienne Fieux$^{\,a,}$\footnote{fieux@math.univ-toulouse.fr}, 
Bertrand Jouve$^{\,b,}$\footnote{bertrand.jouve@cnrs.fr}}
\vspace{3 mm}
}
%{\small
\centerline{{\small a. IMT (UMR5219), Universit\'e Toulouse 3, CNRS, 
118 route de Narbonne, 31062 Toulouse cedex 9, France}}
\centerline{{\small b. LISST (UMR5193), CNRS, Universit\'e Toulouse 2, 
5 all\'ees Antonio Machado, 31058 Toulouse cedex 1, France}}
%}

\vspace{6 mm}

\begin{abstract}
Given a finite undirected graph $X$, 
a vertex is $0$-dismantlable 
if its open neighbourhood is a cone and $X$ 
is $0$-dismantlable if 
it is reducible to a single vertex by successive 
deletions  of $0$-dismantlable
vertices. By an iterative process, a vertex 
is $(k+1)$-dismantlable 
if its open neighbourhood is $k$-dismantlable
and a graph is $k$-dismantlable if 
it is reducible to a single vertex by successive deletions
 of $k$-dismantlable vertices.
 We introduce a graph family, the cubion graphs,
in order to prove that $k$-dismantlabilities give a strict
hierarchy in the class of graphs whose clique 
complex is non-evasive.
We point out how these higher dismantlabilities are related to the derivability 
of graphs defined by Mazurkievicz and we get a new
characterization of the class of closed graphs he defined.
 By generalising the 
notion of vertex transitivity, we consider the issue of higher dismantlabilities 
in link with the evasiveness conjecture. 
\end{abstract}

\vspace{4 mm}

\section{Introduction}

The transition from a graph to its clique complex is one 
of the many ways for associating 
a simplicial complex to a graph.  
Through the notion of dismantlability, 
it  is possible to develop homotopic notions 
adapted to the framework of finite graphs. 
In this paper  we will 
only discuss the dismantlability of vertices. 
The principle of dismantlability in graphs
is to set a rule that indicates the possibility 
of adding or removing vertices
and two graphs are in the same 
homotopy class if one can switch from one 
to the other by a succession of moves (a move being either 
a vertex addition, or a vertex deletion). 
 
The 0-dismantlability is well known: 
a vertex $x$ is 0-dismantlable 
if its open neighbourhood is a cone. 
This means there is a vertex $y$ adjacent to $x$ 
such that any neighbour 
of $x$ is also a neighbour of $y$ (we say that $x$ 
is dominated by $y$) 
and we know that a graph is 0-dismantlable if, 
and only if, it is cop-win
\cite{quilliot, now-win}. 
From a simplicial point of view, the 0-dismantlability 
of a graph is 
equivalent to the strong-collapsibility of its  
clique complex \cite{fl}. 
\textsl{Strong collapsibility} is  introduced 
by Barmak and Minian
\cite{barmin12} who proved that the strong homotopy type 
of a simplicial complex can be described in terms 
of contiguity classes. Assuming that a vertex of a graph 
is 1-dismantlable if its neighbourhood is 0-dismantlable, 
we obtain 1-dismantlability for graphs and it is 
established in \cite{bfj} that 
two graphs $X$ and $Y$ have the same  1-homotopy type
if, and only if, their clique complexes 
$\cl(X)$ and $\cl(Y)$ have the same simple homotopy type.

The $k$-dismantlabilities for $k\geq 2$ reproduce 
this recursive scheme 
to define increasingly large classes of graphs 
to which this paper is dedicated.
In section 2, the main definitions concerning 
graphs and simplicial 
complexes are recalled with the fact 
(Proposition \ref{prop_link})
that the $k$-dismantlability 
of a graph X is equivalent to the $k$-collapsibility of its 
clique complex $\cl(X)$.
While the notions of 0-homotopy 
and 1-homotopy are very different, it should be 
noted that the contribution of 
the higher dismantlabilities
is not so much at the homotopy  level (Proposition 
\ref{prop-Xk-egale-X1})  as at the level of 
dismantlability classes $D_k$, where $D_k$ is the 
class of all $k$-dismantlable graphs.
Section 3 is devoted to the presentation of a family 
of graphs $(\cubion_n)_{n\in \mN}$
(called \textsl{cubion graphs}) which
shows that $(D_k)_{k \in \mathbb N}$
is an increasing sequence of strict inclusions
(Proposition \ref{echelle}):
 $$\forall\: n \geq 2, ~~\cubion_n 
 \in D_{n-1}\setminus D_{n-2}\:.$$
 We also prove that the existence of a 
$(k+1)$-dismantlable and non  $k$-dismantlable vertex implies the 
presence of a clique of cardinal at least $k+3$
(Proposition \ref{prop-existence-d-une-clique}).
In Section 4, the introduction of the \textsl{parasol} 
graph shows the very importance 
of the order in which vertex dismantlings are 
operated  as soon as one leaves 
the class of 0-dismantlable graphs (Proposition \ref{prop-parasol}). 
Setting $D_{\infty}=\bigcup_{k\geq0}D_k$ and considering 
the 1-skeletons of triangulations of the 
Dunce Hat and the Bing's House, 
we explore the question of graphs 
not in $D_{\infty}$ but for which it is sufficient to add 
some $0$-dismantlable vertices to get into $D_{\infty}$. 
We note that $D_{\infty}$ is the smallest fixed point of 
the derivability operator $\bigtriangleup$ of 
Mazurkiewicz \cite{mazur}. 
The set of  $k$-collapsible simplicial complexes with varying values of $k$ in $\N$ 
is the set of non-evasive complexes \cite{barmin12}. 
Therefore,  the elements of $D_{\infty}$
  will be called \textsl{non-evasive}. 
So, the question of whether a $k$-dismantlable and vertex-transitive graph 
is necessarily a complete graph is a particular case of 
the evasiveness conjecture for simplicial complexes, according to which 
\textsl{every vertex homogeneous and non-evasive simplicial complex is a simplex}. 
In the final section, we introduce the notion of 
\textsl{i-complete-transitive graph} to establish a particular 
case for which the conjecture is valid. 

The study of simplicial complexes appears today in a very wide spectrum of research and 
applications \cite{carlsson,ghrist,SCL19}. Very often, 
these complexes are constructed from finite data to obtain 
information on their structure, for instance by the calculus of Betti 
numbers or homotopy groups.
It should be mentioned that the notion of clique complexes 
(also called \textsl{flag complexes} \cite{kozlov})
seems rather general from the homotopic point of view since the barycentric
subdivision of any complex is a flag complex (and the 1-homotopy type of a complex
and of its barycentric subdivision are the same \cite{bfj}).
From this point of view, the notion of higher dismantlabilities is a contribution 
to the study of homotopic invariants for simplicial complexes 
associated to finite data. 
From another point of view, the notion of  higher dismantlabilities
extends the list of graph families built by adding 
or removing nodes with the condition that the neighbourhoods 
of these nodes check certain properties. 
The first example is probably the family of finite chordal 
graphs which is exactly the family of graphs constructed 
by adding simplicial vertices (i.e. whose neighbourhoods are complete graphs) 
from the point. They can also be characterized as graphs 
that can be reduced to a point by a succession of simplicial 
vertex deletions \cite{dirac}. Bridged graphs \cite{anstee} 
and cop-win graphs \cite{quilliot, now-win} are two other 
examples of graph families that can be iteratively constructed 
 respecting a condition on the neighbourhood 
of the node added at each step.
From the perspective of Topological Data Analysis (TDA), 
it is worthwhile to identify to what extent 
a topological structure depends on local constraints.  In a complex
network for instance, the global topological structure can sometimes 
be highly explained by local interaction configurations. When 
they verify certain properties, these local constraints generate 
a global structure that deviates from classical null models and 
can thus explain particular global phenomena. Understanding these
multi-scale links between local and global structures is now 
becoming a key element in the modelling of complex networks. 
Perhaps the best known model is Barabasi's preferential attachment
\cite{barabasi} where the attachment of a new node to the network 
is done preferentially by the nodes of higher degrees. 
Other examples are hierarchical models obtained 
for example by a local attachment of each node to a subset 
of nodes of a maximal clique \cite{ravasz}. 
 These local-global concerns are in line with older issues, 
but still up-to-date, raised in the context of 
local computation \cite{rosenstiehl, godard, litovsky}. 
So, from an application point 
of view, the notion of higher dismantlablities
could enrich the range of tools available  
in all these fields.

\section{Notations and first definitions}

\subsection{Graphs}\label{section-graphs}

In the following, $X=(V(X),E(X))$ is a finite undirected graph,
without multiple edges or loops. The cardinal 
$\vert V(X) \vert$ is equal to the number of vertices of $X$, at least equal to $1$. We denote by $\N$ (resp. $\N^{\star}$) the set of integers $\{0,1, 2,\cdots\}$ (resp. $\{1, 2,\cdots\}$).
 
 We write $x\sim y$, or sometimes just $xy$, for $\{x,y\} \in E(X)$ 
and $x \in X$ to indicate that $x \in V(X)$.
The closed neighbourhood of  $x$ is  $N_X[x]=\{ y \in X\:,\:x \sim y\} \cup \{x\}$
and $N_X(x) =N_X[x]\setminus\{x\}$ is its open neighbourhood. 
When no confusion is possible, $N_X[x]$ will 
also denote the subgraph induced by $N_X[x]$ in $X$. 
Let $S=\{x_1 ,\cdots ,x_n\}$ be a subset of $V(X)$, we denote by $X[S]$ or 
$X[x_1 ,\cdots ,x_n]$ the subgraph induced by $S$ in $X$. The particular 
case where $S=V(X) \setminus \{x\}$ will be denoted by $X-x$. In the same way, the
notation $X+y$ means that we have added a new vertex $y$ to the graph $X$ 
and the context must make clear the neighbourhood of $y$ in $X+y$. A \emph{clique} 
of a graph $X$ is a complete subgraph of $X$. 
A \emph{maximal clique} $K$ of $X$ is a clique 
so that there is no vertex in $V(X)\setminus V(K)$ 
adjacent to each vertex of $K$. For $n\geq 1$, 
the complete graph (resp. cycle) with $n$ vertices is denoted 
by $K_n$ (resp. $C_n$). The graph $K_1$ with one vertex  
will be called \emph{point} and sometimes noted $pt$. 
The complement $\overline{X}$ of a graph $X$ has the 
same vertices as $X$ and two distinct vertices of $\overline{X}$ are adjacent 
if and only if they are not adjacent in $X$. 

The existence of an isomorphism between two graphs is denoted by $X \cong Y$.
 We say that a graph $X$ is a \emph{cone} with apex $x$ if $N_X[x]=X$.
A vertex $a$ \emph{dominates} a vertex $x \neq a$ in $X$
if $N_X[x] \subset N_X[a]$ 
and we note $ x \dom a$. Note that a vertex is dominated if, and only if, its open neighbourhood is a cone. 
Two distinct
vertices $x$ and $y$ are \emph{twins} if $N_X[x]=N_X[y]$.
We will denote by ${\rm Twins}(X)$ the set of \textsl{twin vertices} of $X$.

In a finite undirected graph $X$, a vertex 
is $0$-dismantlable if 
it is dominated  and $X$ is $0$-dismantlable
 if it exists an order $x_1,\cdots, x_n$ 
of the vertices of $X$ such that $x_k$ is 0-dismantlable 
in $X[x_k, x_{k+1}, \cdots, x_{n}]$ for $1\leq k \leq n-1$. 
In \cite{bfj}, we have defined a weaker version of dismantlability. 
A vertex $x$ of $X$ is $1$-dismantlable\footnote{\label{note-S}
In \cite{bfj}, a \textsl{1-dismantlable} vertex was called \textsl{s-dismantlable} 
and \textsl{1-homotopy} was called \textsl{$s$-homotopy}.} 
if its open neighbourhood $N_X(x)$ is a $0$-dismantlable graph.  

Generalising the passage from $0$-dismantlability to $1-$dismantlability, 
the higher dismantlabilities
in graphs are defined iteratively by:
 
\begin{definition} \hspace{1mm}
\newline
\indent $\bullet$ The family $C$ of cones (or conical graphs) is also denoted by $D_{-1}$
and we will say that the cones are the graphs which are $(-1)$-dismantlable.

$\bullet$ For any integer $k\geq 0$, a vertex of a graph $X$ is called $k$-dismantlable  
if its open neighbourhood is  $(k-1)$-dismantlable. The graph $X$ is 
$k$-dismantlable if it is reducible to a vertex by successive deletions of
$k$-dismantlable vertices.  We denote by $D_k(X)$ the set of $k$-dismantlable 
vertices of a graph $X$ and by $D_k$ the set of  $k-$dismantlable graphs.
\end{definition}

\noindent A cone is a $0$-dismantlable graph, that is $D_{-1}\subset D_0$,
%then a $0$-dismantlable vertex is also $1$-dismantlable 
and  
by induction on $k$, we immediately get:

\begin{prop}\label{inclusion}
$\forall k\in \N, D_{k-1} \subset D_{k}$. 
\end{prop}

If $x \in D_k(X)$, we will say that the graph $X-x$ is obtained from 
the graph $X$
by the $k$-deletion of the vertex $x$ and that the graph $X$ is obtained from
the graph $X-x$ by the $k$-addition of the vertex $x$. 
We write $X \rd k Y$ or  $Y \ep k X$
when $X$ is $k$-dismantlable to a subgraph  $Y$, i.e.:
$$X \rd k X-x_1 \rd k X-x_1 -x_2 \rd k \cdots \rd k X-x_1-x_2 - \cdots -x_r=Y$$
with $x_i \in D_k(X-x_1-x_2-\cdots-x_{i-1})$. 
The sequence $x_1,\cdots,x_r$ is called a {\sl $k$-dismantling sequence}. 
The notation $X \rd k pt$ signifies that $X \in D_k$.
A graph $X$ is \textsl{$k$-stiff} when $D_k(X)=\emptyset$.  
We denote by
$D_{\infty}=\bigcup_{k\geq0}D_k$  
the family of graphs which are $k$-dismantlable for some integer $k \geq 0$. 
Cycles of length greater or equal to $4$ and non-connected graphs 
are two examples of graphs which are not in $D_{\infty}$. 
Finally, we write $[X]_k=[Y]_k$ when it is possible 
to go from $X$ to $Y$ by a succession of additions or deletions 
of $k$-dismantlable vertices. Note that $[X]_k$ is an equivalence class.
Two  graphs $X$ and $Y$ such that $[X]_k=[Y]_k$ will be 
said \textsl{$k$-homotopic}$^{\rm \ref{note-S}}$.
We note that for any integers $k\geq 0$ and $k'\geq 0$, any  graph $X$,
any vertex $x$ of $X$ and any vertex $y$ not in $X$, we have
the following \textsl{switching property}:
$$ (\dag)~~~~~~~~~~~~~~
\text{if }~~ X \rd k X-x \ep {k'} (X-x)+y ~~\text{ then }~~
X \ep {k'} X+y \rd k (X+y)-x.$$
Actually, since $x \not\sim y$, this property results from $N_X(x)=N_{X+y}(x)$ 
and $N_{X+y-x}(y)=N_{X+y}(y)$. In particular,
this implies that two graphs  $X$ and $Y$ are $k$-homotopic 
if, and only if, there exists a graph $W$
such that  $X \ep k W \rd k Y$. Nevertheless, the
notion of $k$-homotopy classes is not so relevant 
(see Proposition \ref{prop-Xk-egale-X1}).

\begin{remark}\label{rmk-on-peut-pas-echanger}
Let us also note that the reverse implication of $(\dag)$
is false (see. 
Fig.\ref{figure_bad-exchange}
for a counterexample).

\begin{figure}[ht]
\begin{center}
\includegraphics[width=10cm]{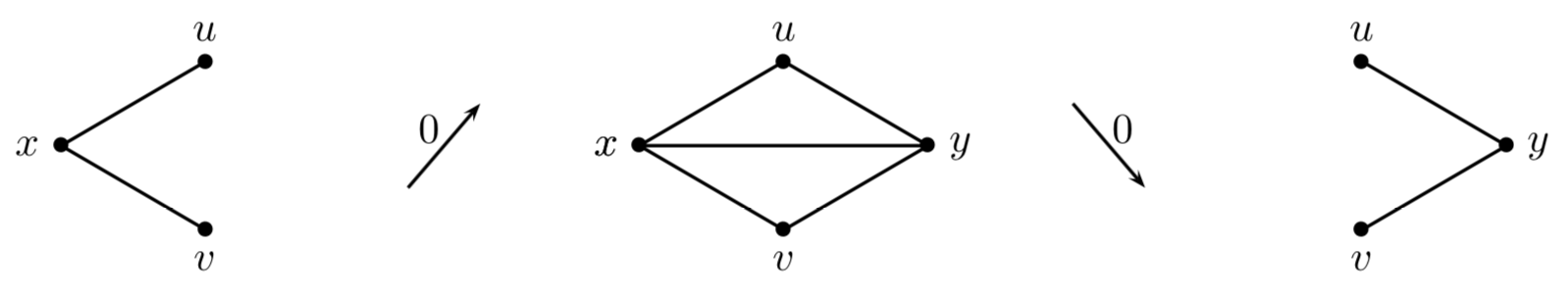}
\caption{Let $Y$ be the 2-path $uxv$:
$Y \ep 0 Y+y \rd 0 (Y+y)-x$ 
but $Y \rd k Y-x$ is impossible for~any~$k$. }
\label{figure_bad-exchange}
\end{center}
\end{figure}

\end{remark}

\subsection{Simplicial complexes}\label{section-simplicial-cxs}

For general facts and references on simplicial complexes, 
see \cite{kozlov}. We recall that a \textsl{ finite abstract
simplicial complex}  $\cK$ is given  by a finite set of vertices 
$V(\cK)$ and a collection  of subsets $\Sigma (\cK)$
of $V(\cK)$ stable by deletion of elements: 
if $\sigma\in \Sigma(\cK)$
and $\sigma'\subset \sigma$, then $\sigma' \in \Sigma(\cK)$.
The elements of $\Sigma(\cK)$ are the \textsl{simplices}
of $\cK$.  If $\sigma$ is a simplex
of cardinal $k \geq 1$, then its dimension is $k-1$ 
and the dimension of $\cK$ 
is the maximum dimension of a simplex of $\cK$. 
The $j$-skeleton of $\cK$
consists of all simplices of dimension $j$ or less. 

Let us recall that for a simplex $\sigma$ of a finite 
simplicial complex $\cK$,
${\rm link}_{\cK}(\sigma)=\{\tau \in \cK \:, \sigma \cap \tau=\emptyset  
~{\rm and}~\sigma \cup \tau \in \cK\}$ is a sub-complex of $\cK$ and 
${\rm star}_{\cK}^0(\sigma)=\{\tau \in \cK \:, \sigma \subset 
\tau\}$ is generally not a sub-complex of $\cK$. 
If $\tau$ and $\sigma$ are two simplices of $\cK$, we say
that $\tau$ is a \textsl{face} (resp. a \textsl{proper face}) 
of $\sigma$  if $\tau \subset \sigma$ 
(resp. $\tau \subsetneq \sigma$).
An \textsl{elementary simplicial collapse} is the suppression 
of a pair of simplices $(\sigma,\tau)$ such that 
$\tau$ is a proper maximal face of $\sigma$ and $\tau$ is not the
face of another simplex (one says that $\tau$ is
a \textsl{free face} of $\cK$).
We denote by $\cK-x$ the sub-complex of $\cK$ induced by the vertices 
distinct from $x$. As defined in \cite{barmin12}, an 
\textsl{elementary  strong collapse}  (or $0$-collapse) in $\cK$ 
is a suppression of a  vertex $x$ such that ${\rm link}_{\cK}(x)$ 
is a simplicial cone.  There is a strong collapse from $\cK$ 
to $\cL$ if  there exists a sequence of elementary strong 
collapses that changes $\cK$  into $\cL$; in that case we 
also say that $\cK$ 0-collapses to
$\cL$.  A simplicial complex is 0-collapsible or
strong collapsible if it 0-collapses to a point. 
By induction, for any integer $k\geq 1$, a vertex of $\cK$
is $k$-collapsible  if ${\rm link}_{\cK}(x)$ is 
$(k-1)$-collapsible. There is a $k$-collapse from $\cK$ to $\cL$ if 
there exists a sequence of elementary $k$-collapses that changes $\cK$ 
into $\cL$ and, in that case, both complexes have the same simple homotopy type. A simplicial complex is $k$-collapsible  if it
$k$-collapses to a point.

Let also recall that a simplicial
complex is \textsl{non-evasive} if it is $k$-collapsible for 
some $k\geq0$ \cite[Definiton 5.3]{barmin12}. A  not non-evasive complex is called \textsl{evasive}. 

When considering graphs, simplicial complexes 
arise naturally by the way of flag complexes.
For any graph $X$, we denote by $\cl(X)$ the abstract simplicial complex
such that $V(\cl(X))=V(X)$ and whose simplices are the subsets of $V(X)$ 
 which induce
a clique of $X$. The simplicial complex $\cl(X)$ is called
the \textsl{clique complex} of $X$ and clique complexes are also called
\textsl{flag complexes} \cite{kozlov}.
A flag complex $\cK$ is completely determined by its 1-skeleton  
(in other words, every flag complex  is the clique complex of its 
1-skeleton) and a simplicial complex $\cK$ is a flag complex 
if, and only if,
its minimal non-simplices are of cardinal 2. Remind that a 
non-simplex of $\cK$ is
a subset of $V(\cK)$  which is not a simplex of $\cK$ and so a
non-simplex $\sigma \subset V(\cK)$ is minimal if all proper subsets of
$\sigma$ are simplices of $\cK$.

 Given a vertex $x$ of a graph $X$, by definition we 
 have ${\rm link}_{\cl(X)}(x)=\cl(N_X(x))$. So, it is 
 easy to observe that a graph $X$ is in $D_0$ if and only 
 if $\cl(X)$ is $0$-collapsible \cite[Theorem 4.1]{fl} and more generally:
 
\begin{prop}\label{prop_link}
For all integer $k\geq 0$,
$X \in D_k$ if, and only if, $\cl(X)$ is $k$-collapsible.
\end{prop}

\noindent So,
by Proposition \ref{prop_link}, the set of non-evasive 
flag complexes is in one to one correspondence with $D_{\infty}$.
Before closing this section, it is important to note that 
since $k$-collapses don't change 
the simple homotopy type:

\begin{prop}\label{prop-Xk-egale-X1}
For all integer $k\geq 1$, $[X]_1=[X]_k$.
\end{prop}

\begin{proof}\label{prop-classe1=classek}
Of course, a graph 1-homotopic to $X$ is also $k$-homotopic to $X$.
Now, let $Y$ be a graph $k$-homotopic to $X$.
The clique
complexes $\cl(X)$ and $\cl(Y)$ have the same simple
simplicial homotopy type 
and, by \cite[Theorem 2.10]{bfj} 
where $[X]_1$ is denoted by $[X]_s$ and $\cl(X)$ is 
denoted by $\Delta(X)$, this implies $[X]_1=[Y]_1$. In particular, 
$Y$ is 1-homotopic to $X$ and, finally, $[X]_1=[X]_k$.
\end{proof}

\section{A hierarchy of families}

\subsection{The family of cubion graphs $(\cubion_n)_{n\in \N}$}

From Proposition \ref{prop_link}, we know that if a graph $X$ 
is $k$-dismantlable for some $k$,
then $\cl(X)$ is a non-evasive simplicial complex.
It is also known \cite{bjorner, kozlov}
that non-evasive simplicial complexes are collapsible and,
\textit{a fortiori}, contractible in the usual topological sense
when the simplicial complex is considered as a topological space 
by the way of some geometrical realisation. 
In particular, this means that a graph whose clique complex
is not contractible cannot be $k$-dismantlable whatever is the integer $k$:

\begin{lemma}\label{lemme-clef}
Given $X_0 \subset X$ and $X \rd k X_0$ for $k\geq 0$, 
if $\cl(X_0)$ is non-contractible, so is $\cl(X)$ 
and $X \notin D_{\infty}$.
\end{lemma}

Let us now show that the inclusions in Proposition \ref{inclusion} 
are strict. 

\begin{definition}{\bf [n-Cubion]}
$\forall n\in\mN$, the \emph{n-Cubion} is the graph 
$\cubion_n$ with vertex set $V(\cubion_n)=
\{\a _{i,\epsilon}, i=1,\cdots,n {\rm ~and~} 
\epsilon = 0,1\} \cup \{x=(x_1,\cdots,x_n), x_i =0,1\}$ 
and edge set $E(\cubion_n)$ defined by: \\
$\bullet \: \forall i \neq j, \forall \epsilon, \epsilon' 
\in \{0,1\}, \a_{i,\epsilon} \sim \a_{j,\epsilon'} $\\
$\bullet \: \forall x \neq x', x \sim x'$\\
$\bullet \: \forall i, \a_{i,1} \sim 
(x_1,\cdots,x_{i-1},1,x_{i+1},\cdots,x_n)$ 
{\rm and} 
$\a_{i,0} \sim (x_1,\cdots,x_{i-1},0,x_{i+1},\cdots,x_n)$
\end{definition}

The $n$-Cubion has $2^n+2n$ vertices partitioned 
into two sets such that: 
$$\cubion_n[\a_{1,0},\a_{1,1},\cdots,\a_{n,0},\a_{n,1}] 
\cong \overline{nK_2} \, \text{  and  } \, 
\cubion_n[x, x\in \{0,1\}^n] \cong K_{2^n}.$$
The $n$-cubion
is built from the \textsl{$n$-hypercube}  
with vertices the $n$-tuples 
$x=(x_1, \cdots, x_n)$ $\in \{0,1\}^n$, each one 
connected to all the others, 
by adding $2n$ vertices $\a _{i,\epsilon}$ 
which induce an \textsl{$n$-octahedron} $\overline{nK_2}$ and 
each $\a _{i,\epsilon}$ is the apex of a cone whose base is the
$(n-1)$-face of the hypercube given by $x_i=\epsilon$. 
This definition gives an iterative process to construct
$\cubion_{n+1}$ from $\cubion_n$.

One sees that $\cubion_1\cong P_4$ the path of length $3$ 
and, clearly, $\cubion_1 \in D_0\setminus D_{-1}$.
The cubion $\cubion_2$ represented in Fig.  \ref{CCC} is in 
$D_1 \setminus D_0$. Indeed, $D_0(\cubion_2)=\emptyset$ 
but $\cubion_2 \rd 1 \cubion_2 - \a_{1,0}$ and  
$\cubion_2 - \a _{1,0} \in D_0$.

%%%%%%%%%%%%%%%%%%%%%%
\begin{figure}
\begin{center}
\includegraphics[width=15cm]{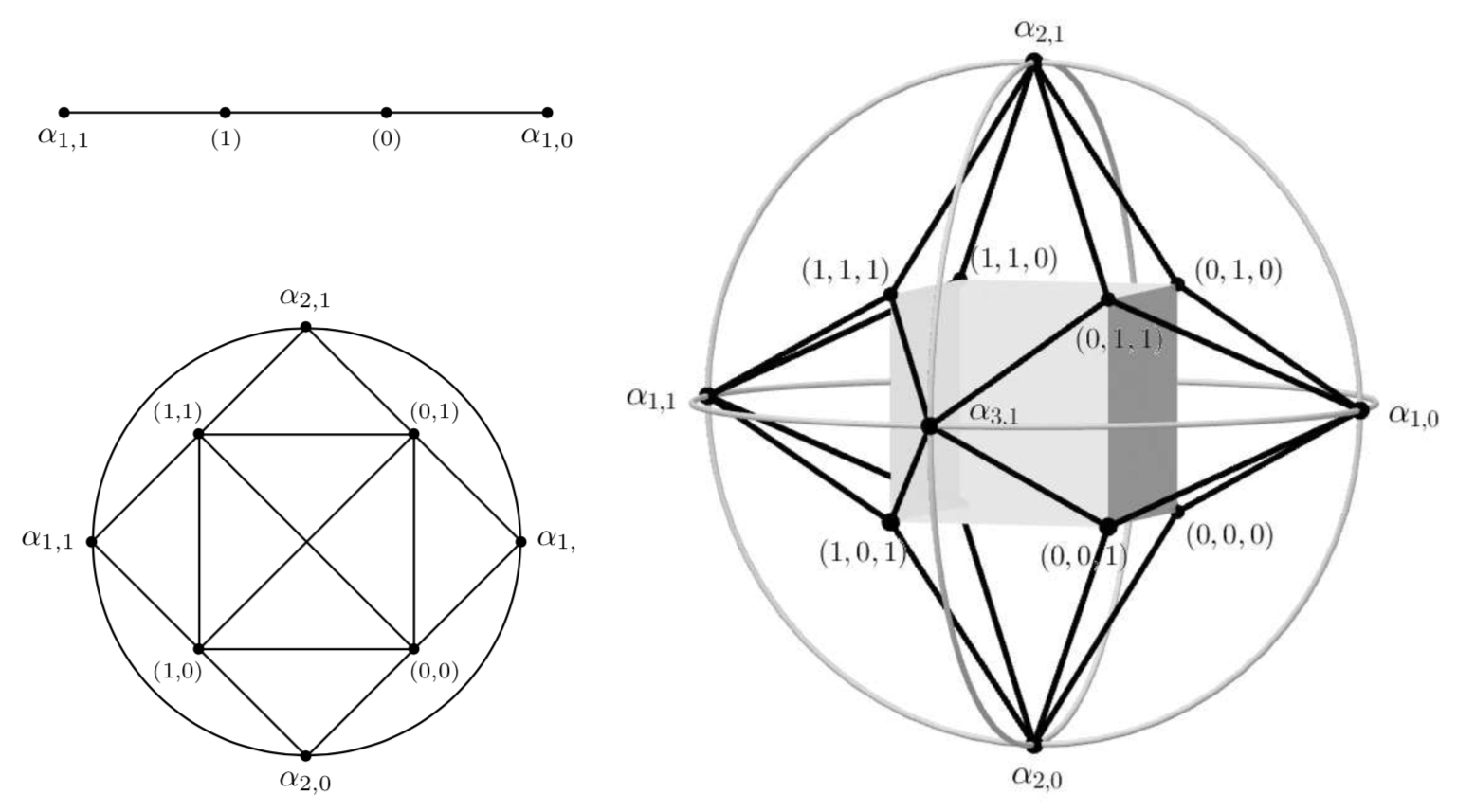}
\caption{(top left) $\cubion_1 \in D_0 \setminus D_{-1}$,  (bottom left) $\cubion_2\in D_1 \setminus D_0$, (right) $\cubion_3\in D_2 \setminus D_1$. The drawing of the $3$-cubion is a perspective view where the central clique $K_8$ is symbolized by a cube: edges of the $K_8$ (ie. between $x$-type vertices) are not drawn, edges between $x$-type and $\alpha$-type vertices are in black, and edges between $\alpha$-type vertices are in grey.  }
\label{CCC}
\end{center}
\end{figure}
%%%%%%%%%%%%%%%%%%%%%%%

\noindent More generally, we get: 

\begin{prop} \label{echelle}
$\forall n \geq 2, \cubion_n \in D_{n-1}\setminus D_{n-2}$.
\end{prop}

\begin{proof}

%\begin{enumerate}
\vspace{.15cm}
\noindent 1. Let us first prove that for any $i$, $\epsilon$ and $x$,
$N_{\smallcubion_n}(\a_{i,\epsilon}) \cong \cubion_{n-1}$ 
and  $N_{\smallcubion_n}(x) \rd 0  \overline{nK_2}$.

\noindent $\bullet$ 
$N_{\smallcubion_n}(\a_{i,\epsilon}) \cong \smallcubion_{n-1}$: 
on one hand, given $i$ and $\epsilon$, the vertex
$\a_{i,\epsilon}$ is linked to all the $\a_{j,\epsilon'}$ 
except when $i=j$. Thus,
$N_{\cubion_n[\a_{1,0},\a_{1,1},\cdots,\a_{n,0},\a_{n,1}]}(\a_{i,\epsilon}) 
\cong \overline{(n-1)K_2}$.
On the other hand, within the set of the $n$-tuples 
$x=(x_1,\cdots,x_n)$, $\a_{i,\epsilon} \sim
(x_1,\cdots,x_{i-1},\epsilon,x_{i+1},\cdots,x_n)$. 
These $2^{n-1}$ vertices $x$ whose $i^{th}$ entry
is fixed and equal to $\epsilon$ are  all linked 
together and thereby induce a subgraph isomorphic 
to $K_{2^{n-1}}$ in $\cubion_n$. The edges between
$\overline{(n-1)K_2}$ and $K_{2^{n-1}}$ 
are inherited from $\cubion_n$ and thus
$N_{\smallcubion_n}(\a_{i,\epsilon}) \cong \cubion_{n-1}$.

\noindent $\bullet$ $N_{\cubion_n}(x) \rd 0 \overline{nK_2}$: 
among all the $\a_{i,\epsilon}$  the vertex $x=(x_1,\cdots,x_n)$ 
is linked exactly to the $n$ vertices
$\a_{1,x_1}, \cdots, \a_{n,x_n}$.
Let $X=N_{\smallcubion_n}(x) \setminus 
\{\a_{1,x_1}, \a_{2,x_2}, \cdots, \a_{n,x_n}\}$, a
partition of $X$ is given by $X^0 \cup X^1 \cup \cdots \cup X^{n-1}$ 
with $X^k=\{y \in X,y \text{ is linked to exactly } k 
\text{ vertices } \a_{i,x_i}\}$. 
Clearly $X^i$ has $\binom{n}{i}$ elements. 
For example, we  have $X^0=\{(1-x_1,1-x_2,\cdots,1-x_{n-1},1-x_n)\}$
and $X^{n-1}=\{\hat{x_i}, i=1,\cdots,n\}$  with
$\hat{x_i}=(x_1,x_2,\cdots,x_{i-1},1-x_{i},x_{i+1},
\cdots, x_{n-1}, x_n)$.  For any $y\in X \setminus X^{n-1}$,  
there exist $i \neq j$,  such that 
$y_i=1-x_i$ and $y_j=1-x_j$. Hence, $y$ is dominated by $\hat{x_i}$ 
and $\hat{x_j}$ both in $N_{\smallcubion_n}(x)$. 
By successive $0$-dismantlings  of the vertices $y$, 
we obtain $N_{\smallcubion_n}(x) \rd 0 X^{n-1} \cup \{\a_{1,x_1},
\a_{2,x_2},  \cdots, \a_{n,x_n}\}$. Finally, just notice that,  
between the vertices of $X^{n-1} \cup 
\{\a_{1,x_1}, \a_{2,x_2}, \cdots, \a_{n,x_n}\}$, all 
the possible edges exist except the 
$\hat {x_i} \a_{i,x_i}$ 
and thus $X^{n-1} \cup \{\a_{1,x_1}, \a_{2,x_2}, 
\cdots, \a_{n,x_n}\} \cong \overline{nK_2}$.  

\vspace{.15cm}
\noindent 2. By induction on $n$, 
$\cubion_{n-1} \in D_{n-2}\setminus D_{n-3}$ and
as we have proven that 
$N_{\smallcubion_n}(\a_{i,\epsilon}) \cong \cubion_{n-1}$, 
$\a_{i,\epsilon} \in D_{n-1}(\cubion_n) \setminus
D_{n-2}(\cubion_n)$. 
Moreover, since the simplicial complex $\cl(\overline{nK_2})$ 
is non-contractible because it is a triangulation of 
the sphere $S^{n-1}$, Lemma \ref{lemme-clef}
implies $N_{\cubion_n}(x) \notin D_{\infty}$
and thus  $x \notin D_{n-2}(\cubion_n)$. 
Therefore $D_{n-2}(\cubion_n)=\emptyset$ and 
$\cubion_n \notin D_{n-2}$. Now, 
$$\cubion_n \rd {n-1} \cubion_n - \{\a_{n,0},\a_{n,1}\}$$ since
$\a_{n,0}$ and $\a_{n,1}$ are $(n-1)$-dismantlable and not linked. 
In $\cubion_n - \{\a_{n,0},\a_{n,1}\}$, note that
$(x_1,\cdots,x_{n-1},0)$ and $(x_1,\cdots,x_{n-1},1)$ are twins 
and therefore 
$$\cubion_n - \{\a_{n,0},\a_{n,1}\} \rd 0 
\cubion_n - \bigl\{\a_{n,0},\a_{n,1}, 
(x_1,\cdots,x_{n-1},0);\: (x_1,\cdots,x_{n-1})\in \{0,1\}^{n-1}
\bigr\}\cong
\cubion_{n-1}.$$
By induction hypothesis, $\cubion_{n-1} \in 
D_{n-2} \subset D_{n-1}$. Finally, $\cubion_n \in D_{n-1}$. 
\end{proof}

\vspace{.15cm}
Propositions \ref{inclusion} and \ref{echelle} 
now give the following theorem: 

\begin{theorem}\label{sequence}
The sequence $(D_k)_{k \geq 0}$ is strictly increasing: 
$$D_{-1} \subsetneq D_0 \subsetneq D_1 \subsetneq D_2 \subsetneq \cdots
\subsetneq D_k \subsetneq D_{k+1} \subsetneq \cdots $$
\end{theorem}

There are no graphs with fewer vertices than $\cubion_1$ 
in $D_0 \setminus D_{-1}$. One can verify the same result 
for $\cubion_2$ in $D_1\setminus D_0$, but there are 
graphs in $D_1\setminus D_0$ with fewer edges.

\subsection{Critical $k$-dismantlability}

Let's complete this section with results
on graphs  in $D_k \setminus D_{k-1}$  with $k\geq 1$. 
Such  a graph $X$  does not always have a vertex 
in  $D_k(X)\setminus D_{k-1}(X)$. 
Indeed, by duplicating each vertex of a graph
in $D_k\setminus D_{k-1}$ with a twin, we get
a new graph also  in $D_k\setminus D_{k-1}$
in which each  vertex 
is $0$-dismantlable and, so, 
is not in $D_k(X)\setminus D_{k-1}(X)$.
However we have the following result:

\begin{lemma}\label{lem-DkMaisPasDk-1}
Given $k\in \N^{\star}$ and $X \in D_k \setminus D_{k-1}$, 
there exists $x\in V(X)$ and $Y$ an induced subgraph of $X$ 
such that  $x\in D_{k}(Y)\setminus D_{k-1}(Y)$.
\end{lemma}

\begin{proof}
Set $V(X)=\{x_1,\cdots, x_n\}$ and suppose that
$x_1,\cdots,x_{n-1}$  is a $k$-dismantling sequence 
from $X$ to the point $x_n$. By definition, 
$\forall i\in \{1, \cdots, n-1\}$, $x_i \in D_k(X[x_i,x_{i+1}, 
\cdots, x_{n}])$.  Since $X \not\in D_{k-1}$,  
the sequence $x_1,\cdots,x_{n-1}$ is not a  $(k-1)$-dismantling
sequence of $X$. Therefore, there exists $i_0 \in \{1,\cdots,n-1\}$
such that $x_{i_0} \notin D_{k-1}(X[x_{i_0},x_{i_0+1}, \cdots, x_{n}])$, 
i.e. $x_{i_0} \in D_{k}(Y)\setminus D_{k-1}(Y)$ where
$Y=X[x_{i_0},x_{i_0+1}, \cdots, x_{n}]$.
\end{proof}

We remark that any  connected graph
with at most three vertices contains at least one 
apex and therefore
any $X \in D_0\setminus D_{-1}$ 
has at least four vertices. 
We recall that  the
\textsl{clique number} $\omega(X)$ of a graph $X$ 
is the maximum number of vertices  of a clique of $X$.

\begin{prop}\label{prop-existence-d-une-clique}
Given $k\in \N^{\star}$, if $D_k(X)\setminus D_{k-1}(X)
\neq \emptyset$,  then $\omega(X) \geq k+2$.
Moreover, if $x \in D_k(X)\setminus D_{k-1}(X)$, there
is a clique with $k+2$ vertices and containing $x$. 
\end{prop}
\begin{proof}
The proof is by induction on $k$.

For $k=1$, if there exists $x_2 \in D_1(X)
\setminus D_0(X)$,  then $N_X(x_2)\in D_0\setminus C$.
  Since $N_X(x_2)$ is not a
cone but is $0$-dismantlable, it contains an edge, so $X$ contains a triangle.

Now, let $X$ be a graph such that $D_{k+1}(X)\setminus
D_{k}(X)\neq \emptyset$  and denote by $x_{k+2}$ a vertex 
such that $N_X(x_{k+2})\in D_{k}\setminus D_{k-1}$. 
From Lemma \ref{lem-DkMaisPasDk-1}, there exists 
$x_{k+1} \in V(N_X(x_{k+2}))$ and  $Y$ an induced subgraph of
$N_X(x_{k+2})$ such that $x_{k+1}\in D_{k}(Y)\setminus
D_{k-1}(Y)$. The induction hypothesis  applied to $Y$ gives
that $Y$ contains an induced subgraph $K\cong K_{k+2}$. 
As $Y \subset N_X(x_{k+2})$,  $K+x_{k+2}$ is a 
complete subgraph of $X$ with $k+3$ vertices.

And it follows from this proof that any vertex in $D_k(X)\setminus D_{k-1}(X)$ is in 
a clique of $X$ of cardinal $k+2$.
\end{proof}

A direct consequence of  Lemma \ref{lem-DkMaisPasDk-1}
and Proposition \ref{prop-existence-d-une-clique} is:

\begin{cor}\hspace{.5cm}
\newline
\indent (i) Given $k\in \N$, 
if $X \in D_k\setminus D_{k-1}$, 
then $\omega(X) \geq k+2$.

(ii) If $X \in D_{\infty}$ and $\vert V(x)\vert =n$, then 
$X \in D_{n-2}$.
\end{cor}

Thus, if a graph of $D_{\infty}$ contains no triangle, 
it is in $D_0$ and it is not hard to prove by induction 
that a $0$-dismantlable graph without a triangle is a tree. 
So, the only graphs of $D_{\infty}$  
without a triangle are the trees.
A more directed proof of 
this fact is obtained by considering the clique complexes.
Indeed, if a graph $X$ is triangle-free, then \cl(X) is a 
$1$-dimensional complex, and if $X$ is in $D_{\infty}$, 
Proposition \ref{prop_link} implies that \cl(X)  is $k$-collapsible.
Thus, \cl(X) has to be a tree and so is $X$.  

\section{Some results on $D_{\infty}$}
\subsection{Order in dismantlabilities}
For $0$-dismantlability, the order of dismantlings does not 
matter and therefore the $0$-stiff graphs to which a graph $X$ 
is $0$-dismantlable are isomorphic  (\cite[Proposition 2.3]{fl},
\cite[Proposition 2.60]{hell-nesetril}). 
This property is no longer true for $k$-dismantlability 
with $k\geq 1$. The graph $X$ of Fig. \ref{figure_stiff}   gives a simple 
example of a graph that is
$1$-dismantlable either to $C_4$, 
or to $C_5$  (depending 
on the choice and order of the vertices to
$1$-dismantlable)
which are non-isomorphic $1$-stiff graphs.

\begin{figure}[ht]
\begin{center}
\includegraphics[width=3cm]{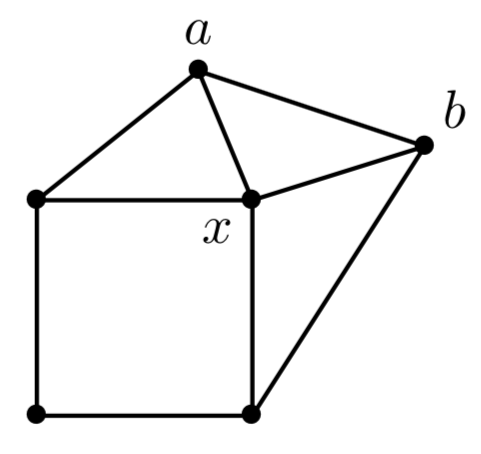}
\caption{$X \rd 1 X -a-b\cong C_4$ and $X \rd 1 X -x \cong C_5$ }
\label{figure_stiff}
\end{center}
\end{figure}

\noindent Actually, there is an important gap between
$0$-dismantlability and $k$-dismantlability with $k\geq 1$. 
We have already noted in Proposition \ref{prop-Xk-egale-X1}
that, for any graph $X$ and any $k\geq 1$,
$[X]_k=[X]_1$ while the inclusion $[X]_0\subset [X]_1$
of homotopy classes is strict:
\begin{itemize}
    \item $[C_4]_0\neq [C_5]_0$ because the cycles $C_4$ and $C_5$
    are non isomorphic 0-stiff graphs. 
    \item $[C_4]_1=[C_5]_1$ as it is shown by graph $X$
    in Fig. \ref{figure_stiff}.
\end{itemize}

%\psset{unit=1.2cm}
\begin{figure}[!h]
\begin{center}
\includegraphics[width=13cm]{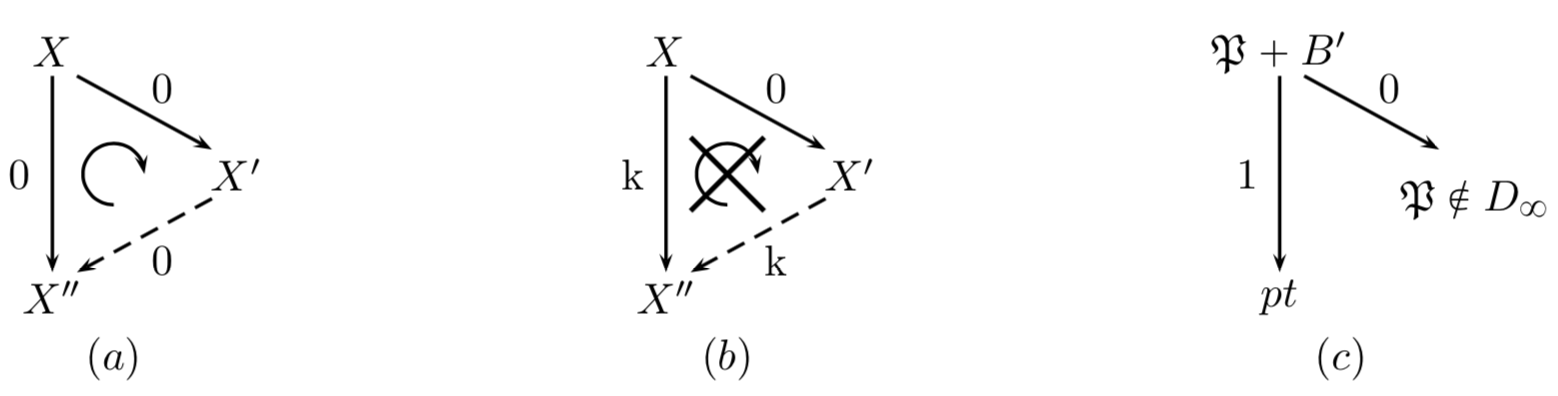}
\caption{With $X''$ an induced subgraph of $X'$
and $X'$ an induced subgraph of $X$: $(a)$ the dashed arrow always exists, $(b)$ the dashed arrow may exist or not, $(c)$ an illustration of $(b)$ where the dashed arrow does not exist.}
\label{commutative-or-not-triangles}
\end{center}
\end{figure}

A major fact concerning the difference between 0-dismantlability
and $k$-dismantlability for $k\geq 1$ is
that  (\cite[Corollary 2.1]{fl})
$$(X\rd 0 X'' \;\; 
, \;\; X \rd 0 X' \;\; \text{and} \;\;X''\subset X')  ~~ 
\Longrightarrow ~~ X' \rd 0 X''$$
while, for $k\geq 1$, in general 
(cf. Fig. \ref{commutative-or-not-triangles}, $(b)$ and $(c)$):
$$(X\rd k X'' \;\; 
, \;\; X \rd 0 X' \;\; \text{and} \;\;X''\subset X') ~~ 
\diagup \hspace{-.45cm}\Longrightarrow ~~ X' \rd k X''.$$
Actually, one can find graphs $X$, $X'$ and $X''$ such that $X'' \subset X' \subset X$, $X\rd 1 X''$,  
$X \rd 0 X'$ and  $X'$ is not $k$-dismantlable to $X''$ for any integer $k\geq 0$. 
To prove this, we introduce the Parasol graph: 
\begin{definition}[Parasol graph]
The \emph{Parasol}
is the graph  $\parasol$ with $15$ 
vertices drawn in Fig.  \ref{parasol}. From $\parasol$, 
we build a graph $\parasol+B'$ by adding to  $\parasol$ 
a vertex $B'$ linked to $B_1$ and to the neighbours of 
$B_1$ except $B_3$ and $B_6$. 
\end{definition}

\noindent The neighbours of the vertices 
of $\parasol$ are as follows, 
for all $i\in \{1,\cdots,n\}$:
\begin{itemize}
\item $N_{\parasol}(B_i)$ is isomorphic 
to $C_4$ with two disjoint pendant edges attached 
to two consecutive vertices of the cycle
\item $N_{\parasol}(A_i) \cong C_5$
\item $N_{\parasol}(I) \cong C_7$
\end{itemize}

\noindent Consequently, $D_{k}(\parasol)=\emptyset$ for all positive integer $k$ and:

\begin{figure}[!h]
\centering
\includegraphics[width=15cm]{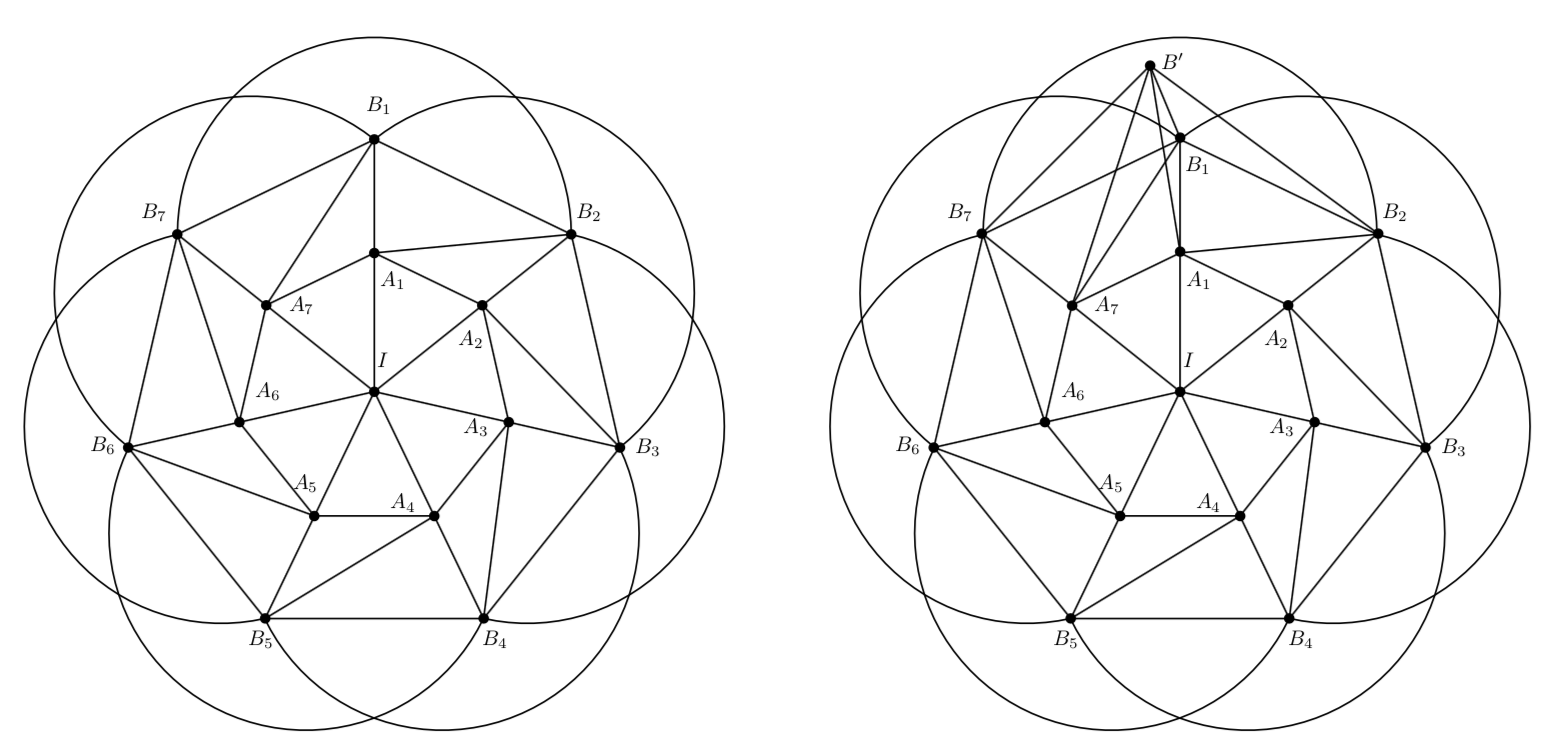}
\caption{(left) the parasol
    graph $\parasol$ and (right) the graph $\parasol+B'$. }
    \label{parasol}
\end{figure}

\begin{prop}\hspace{1cm}\label{prop-parasol}
\smallbreak
(i) $\parasol \notin D_{\infty}$.
\smallbreak

(ii) $\parasol+B' \rd 0 \parasol$.
\smallbreak

(iii) $\parasol+B' \rd 1 (\parasol+B')-B_1 \rd 1 pt$.
\end{prop}

\begin{proof}
(i) $\parasol$ is not  in $ D_{\infty}$ by application of Lemma 
\ref{lemme-clef} because each vertex has a neighbour which is 
a cycle of length at least 5 or which is 0-dismantlable to a cycle
of length 4.

(ii) As $B' \dom B_1$ in $\parasol+B'$, the vertex  $B'$ is $0$-dismantlable in $\parasol+B'$ and $\parasol+B' \rd 0 \parasol \notin D_{\infty}$.

(iii) It is easy to verify that the neighbourhood of $B_1$ in $\parasol +B'$
is a 0-dismantlable graph; so, 
$B_1 \in D_1(\parasol+B')$ and $\parasol+B' \rd 1 (\parasol+B')-B_1$.
Then, following the increasing order of the indexes $i$, all the $B_i$ are successively $1$-dismantlable with
a path as neighbourhood.
The remaining graph induced by $I$ and the 
vertices $A_i$ is a cone and thus $0$-dismantlable. 
\end{proof}

\noindent This example shows 
that, for graphs in $D_k$ with $k\geq1$,
the  dismantling order is crucial: it is possible to reach (resp. quit) 
$D_{\infty}$ just by 
adding (resp. removing) a 0-dismantlable vertex (Fig. \ref{commutative-or-not-triangles}, $(c)$).

%%\label{rmk-ws}
The parasol graph is not in $D_{\infty}$ 
but it is worth noting that the parasol graph is
ws-dismantlable:
$$ \parasol \rd{ws} pt$$
Let us recall (cf. \cite{bfj}) that ws-dismantbility allows not only
1-dismantlability of vertices but also
of edges (an edge $\{a,b\}$ of a graph $X$ is 1-dismantlable 
whenever $N_X(a)\cap N_X(b)\in D_0$)\footnote{In \cite{bfj}, 
\textsl{1-dismantlable edge} was
called \textsl{$s$-dismantlable}.}. 
For example, in the parasol graph, one can 1-delete
 the edge  $\{B_2, B_7\}$
and  the remaining graph is 1-dismantlable
(beginning by $B_1$).
It is well known (\cite[Lemma 3.4]{cyy},\cite[Lemma 1.6]{bfj})
that the 1-dismantlability of an edge
can be obtained by the 0-addition of a 
vertex followed by the 1-deletion of another vertex.
As an illustration, the sequence
$$ \parasol \ep 0 \parasol + B' 
\rd 1 (\parasol +B') -B_1$$
can be seen as 1-deletions of the edges 
$\{B_1, B_6\}$ and $\{B_1, B_3\}$.

Thanks to the switching property $(\dag)$ which 
allows to switch $0$-expansions and 
$1$-dismantlabilities, we get:
$$ X \text{ ws-dismantlable} \Rightarrow \exists W\,\, 
\text{such that} \,\, X \ep 0 W \rd 1 pt.$$
The question of which other graphs this property 
extends to is open. More precisely, while any graph $X$ 
for which there exists $W$ such that $X \ep 0 W \rd 1 pt$ 
has a contractible clique complex $\cl(X)$, the reverse 
implication 
remains open: if $X$ is a graph such that $\cl(X)$ 
is contractible, does it exist $W$ such that $X \ep 0 W \rd 1 pt$ ?
 The graphs $DH$ (Fig. \ref{DH-figure}) 
and $BH$ (Fig. \ref{BH-figure} and \cite[Fig. 3, Fig. 4]{cyy}), 
1-skeletons of triangulations of the Dunce Hat 
and the Bing's House respectively, are interesting cases. 
Indeed, $\cl(DH)$ and $\cl(BH)$ are known to be 
contractible but non collapsible and this implies that
both graphs are neither in $D_{\infty}$ nor in the set 
of ws-dismantlable graphs. However:

\begin{prop}\label{BH_DH}
There exist two graphs $W_{DH}$ and $W_{BH}$ such that: 
$$DH \ep 0 W_{DH} \rd 1 pt \,\,\,\, \text{and} \,\,\,\, BH \ep 0 W_{BH} \rd 1 pt   $$
\end{prop}
\begin{proof}
For both graphs, the process is the same and consists in 
successive $0$-additions of vertices so as to transform 
non-$0$-dismantlable neighbourhoods of some vertices into 
$0$-dismantlable ones. For these two graphs $DH$ and $BH$, 
we will transform some cycles into wheels by 0-additions of
vertices. Here we give the sequence of $0$-additions and 
$1$-deletions only for the Dunce Hat and the details 
for the Bing's House are given in Appendix.  
With notations of Fig. \ref{DH-figure}, we do the  
following $0$-additions and $1$-deletions with $G_0=DH$:
\begin{enumerate}
\item Within $N_{G_0}(1)$: $0$-additions of vertices 
$1'$ and $1''$  linked to $1,2,j,i,h$ 
and $1,3,d,e,f$ respectively. Note that $1' \dom 1$ and $1'' \dom 1$. 
Now, since $N_{DH+1'+1''}(1)$ is made of two $4$-wheels 
linked by a path, so is a $0$-dismantlable graph, 
we $1$-delete vertex $1$. Let us note $G_1=DH+1'+1"-1$.
\item Within $N_{G_1}(2)$: $0$-addition of vertex $2'$ 
linked to $2,1',h,4,j$ and $1$-deletion of vertex $2$. 
Let us note $G_2=G_1+2'-2$.
\item Within $N_{G_2}(3)$: $0$-addition of vertex $3'$ 
linked to $3,1'',d,4,f$ and $1$-deletion of vertex $3$. 
Let us note $G_3=G_2+3'-3$.
\smallbreak          
\noindent Now, observe that $G_3 \in D_1$. Indeed, 
\smallbreak 
\item Vertex $4$ is in $D_1(G_3)$ since its 
neighbourhood is the path $bcd3'fgh2'jkl$. 
Let us note $G_4=G_3-4$, vertices $2'$ and $3'$ 
are in $D_1(G_4)$ since $N_{G_4}(2')=h1'j$ and $N_{G_4}(3')=d1''f$ 
are disjoint $2$-paths. Let us note $G_5=G_4-2'-3'$, vertices $1'$ 
and $1''$ are in $D_1(G_5)$ since $N_{G_5}(1')$ and $N_{G_5}(1')$ 
are also disjoint $2$-paths.
\item  Now the resulting graph $G_5-1'-1''$ is a $12$-wheel 
centered in $z$. Like any cone, it is $0$-dismantlable. 
\end{enumerate}  
The switching property $(\dag)$ finishes the proof
and $W_{DH}=DH+1'+1''+2'+3'$.
\end{proof}
\begin{remark}
The strategy used in the previous proof is based on the removal of the vertices $1, 2, 3$ and $4$ corresponding to the gluing data of the Dunce Hat in order to get the $0$-dismantlable $12$-wheel centered in $z$. However, to get a $1$-dismantlable graph, the $0$-additions of vertices $1'$ and $1''$ are enough, as shown by the $1$-dismantling sequence $1, a, b, c, d, e, f, g, h, i$, $j, k, z, l, 1', 2, 4, 3, 1''$ of $DH+1'+1''$ which alternates $0$- and $1$-dismantlings.
\end{remark}

%%%%%%%%%%%%%%%%%%%%%%%%%%
%Duncehat
%%%%%%%%%%%%%%%%%%%%%%%%%

\begin{figure}
\includegraphics[width=16cm]{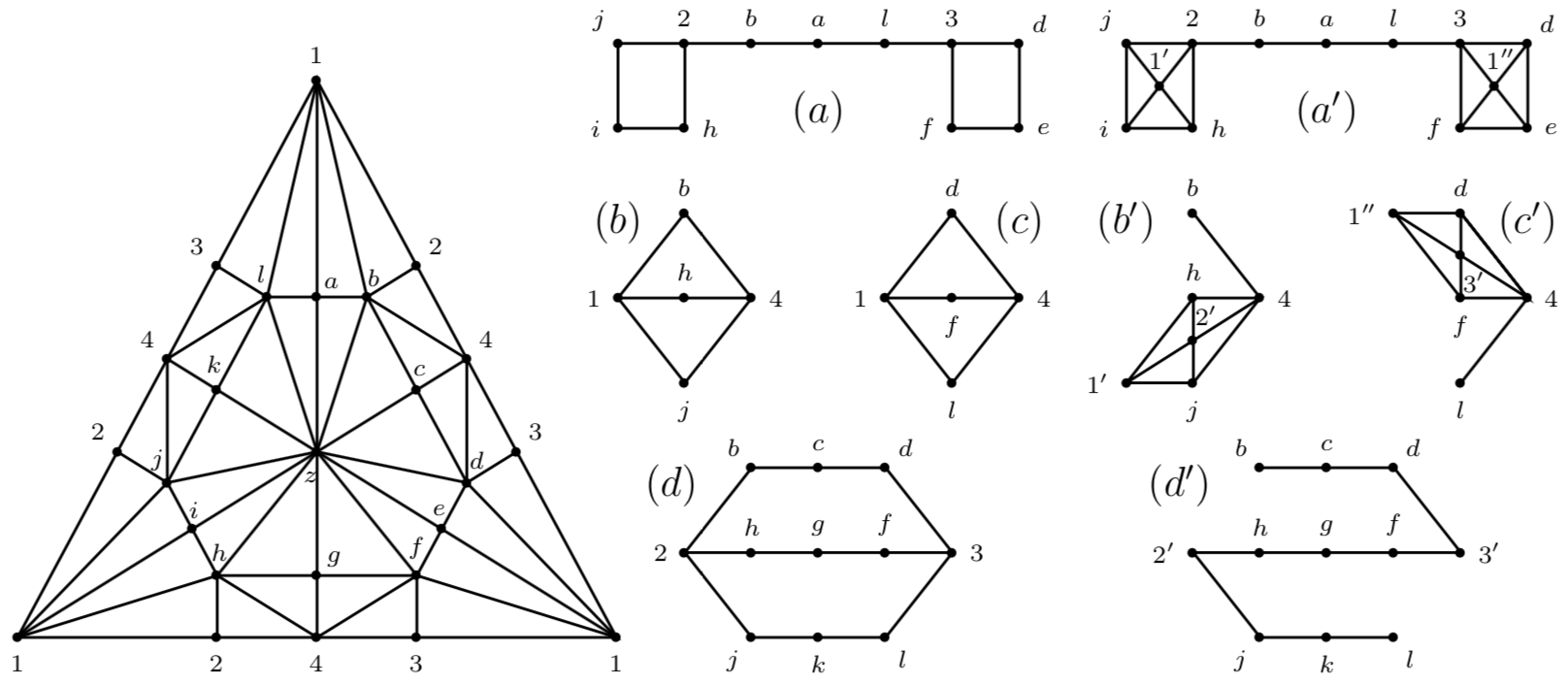}
\caption{(left) The graph $DH$, the 1-skeleton of a 
triangulation of the Dunce Hat. (middle) (a), (b), (c) and (d) 
are the neighbourhoods of vertices $1$, $2$, $3$ and $4$ in $G_0=DH$, 
respectively. (right) (a') neighbourhood of vertex $1$ 
in $G_0+1'+1''$ (b') neighbourhood of vertex $2$ in $G_1+2'$ 
with $G_1=G_0+1'+1''-1$ (c') neighbourhood of vertex $3$ 
in $G_2+3'$ with $G_2=G_1+2'-2$ (d') neighbourhood of 
vertex $4$ in $G_3$ with $G_3=G_2+3'-3$.}
\label{DH-figure}
\end{figure}

\subsection{A link with graph derivability}

In \cite{mazur}, Mazurkiewicz introduces the following notion
of \textsl{locally derivable graphs}.
%\begin{definition}[Definition 2.2 of \cite{mazur}]
For any family  $\mathcal{R}$ of non-empty graphs,
$\bigtriangleup(\mathcal{R})$ is the smallest family 
of graphs containing the point graph $pt$ and such that 
$$\left( X-x \in \bigtriangleup(\mathcal{R}) \;\; 
\text{and} \;\; N_X(x)\in \mathcal{R} \right) \Rightarrow X 
\in \bigtriangleup(\mathcal{R}). $$
Graphs in $\bigtriangleup(\mathcal{R})$ are called 
\textsl{locally derivable by $\mathcal{R}$}.
%\end{definition}
By definition, the graphs of
$\bigtriangleup(\mathcal{R})$ are non-empty 
and connected graphs and $\bigtriangleup$ is monotone:
$\mathcal{R} \subset \mathcal{R'}$
implies $\bigtriangleup(\mathcal{R})\subset 
\bigtriangleup(\mathcal{R}')$. 
By an inductive proof on the cardinals of the vertex sets, it is easy
to see that $D_0=\bigtriangleup(C)$ and more generally
(recall that $C$ is also denoted $D_{-1}$):
\begin{prop}\label{derivability1}
For all $k\in \N$,  $D_{k}$ 
is locally derivable by $D_{k-1}$, i.e. 
$\bigtriangleup(D_{k-1})=D_{k}$.
\end{prop}
It is worth noting the following fact:
\begin{prop}
$D_{\infty}$ is the smallest fixed point of $\bigtriangleup$.
\end{prop}
\begin{proof}
Following the notations of \cite{mazur} and given 
a family $\mathcal{R}$ of graphs, we denote by
$\bigtriangleup^{\star}(\mathcal{R})$ the set 
$\bigcup_{n\geq 0}\bigtriangleup^n(\mathcal{R})$ with the convention 
 $\bigtriangleup^0(\mathcal{R})=\{pt\}$ and
$\bigtriangleup^1(\mathcal{R})=\bigtriangleup(\mathcal{R})$. 
Let us note that, if $\mathcal{F}$ is a fixed family by $\bigtriangleup$, then $\mathcal{F}=\bigtriangleup^{\star}(\mathcal{F})$.
Now, for any family $\mathcal{R}$, the inclusion 
$\{pt\} \subset \bigtriangleup(\mathcal{R})$ 
and the monotony of $\bigtriangleup$ imply  
$\bigtriangleup^{\star}(\{pt\}) \subset 
\bigtriangleup^{\star}(\mathcal{R})$.
Consequently, $\bigtriangleup^{\star}(\{pt\})$
is the smallest fixed point of $\bigtriangleup$
and we have to prove that $D_{\infty}=
\bigtriangleup^{\star}(\{pt\})$.

From Proposition \ref{derivability1}, 
$D_{\infty}=\bigtriangleup^{\star}(C)$ 
and $\bigtriangleup^{\star}(\{pt\}) \subset D_{\infty}$.  
For the reverse inclusion, 
by induction on $n$, observe that  
any cone with $n$ vertices is in $\bigtriangleup^{n-1}(\{pt\})$. 
Consequently, $C\subset \bigtriangleup^{\star}(\{pt\})$
and so $D_{\infty}=\bigtriangleup^{\star}(C)\subset \bigtriangleup^{\star}(\{pt\})$. 
\end{proof}

In \cite{mazur}, the set $D_{\infty}$ is denoted by $F$ 
(and the elements of $D_{\infty}=F$  
are called \emph{closed graphs}) 
and the author states that for any $X \in F$, if  $x\in V(X)$ 
with $N_X(x)\in F$, then $X-x \in F$. 
The graph $\parasol+B'$ is a counter-example. Indeed, 
$\parasol \not \in D_{\infty}$ while, by Proposition \ref{prop-parasol}, 
$\parasol+B' \in D_1 \subset D_{\infty}$ and 
$N_{\parasol+B'}(B') \in D_{\infty}$ because  it is a cone with apex $B_1$.
%    \item while $\parasol \not \in D_{\infty}$

%%%%%%%%%%%%%%%%%%%
%%%%%%%%%%%%%%%%%%%

\section{Vertex-transitive graphs, $k$-dismantlability 
and evasivity}
The relation $\cE$ defined on the set $V(X)$ of vertices of a  graph $X$ by 
$x \:\cE\: y \Longleftrightarrow N_X[x]=N_X[y]$ is an equivalence relation 
whose equivalence classes  are maximal sets of twin vertices. 
With notations of \cite{sab61}, we denote by $X_\star$ the graph obtained 
from this equivalence relation: $V(X_\star)$ is the set of equivalence 
classes of $\cE$ with adjacencies $x_\star \sim y_\star$ if, and only if, 
$x\sim y$. 

\begin{prop}\cite[Lemma 6.4]{sab61} \label{sab61} Let $X$ be a graph.
\vspace{1 mm}

(i) There is a subgraph of $X$ isomorphic to $X_\star$.
\vspace{1 mm}

(ii) $(X_\star)_\star=X_\star$ (i.e., ${\rm Twins}(X_\star)=\emptyset$).
\vspace{1 mm}

(iii) $X\cong X_\star$ if, and only if, $x_\star=\{x\}$
for every vertex $x$ of $X$.
\vspace{1 mm}

(iv) If $X\cong Y$, then $X_\star \cong Y_\star$.
\end{prop}

The following lemma is easy to prove:
\begin{lemma}\label{lem-X-0-dismants-on-Xstar}
$X \in D_0$ if, and only if, $X_\star \in D_0$.
\end{lemma}

\noindent We recall that a graph $X$ is \textsl{vertex-transitive} if 
its automorphism group ${\rm Aut}(X)$ acts transitively on $V(X)$ 
(i.e., for any vertices $x,y$, there is an automorphism $\f$ of $X$ 
such that $\f(x)=y$).  In a vertex-transitive graph, all  vertices 
have  isomorphic neighbourhoods and:

\begin{lemma}\label{lem-vt-imply-twin-egale-D0}
Let $X$ be a vertex-transitive graph.
\smallbreak

(i) ${\rm Twins}(X)=D_0(X)$.
\smallbreak

(ii) $X_\star$ is  vertex-transitive.
\smallbreak

(iii) Let $x \in X$ such that $N_X[x]$ is a clique,  
$x_{\star}$ is equal to $N_X[x]$ and is a connected component of $X$. 
\end{lemma}

\begin{proof}
(i) The inclusion ${\rm Twins}(X)\subset D_0(X)$ is obvious.
Now, let $a$ and $b$ be two vertices with $a \dom b$. 
The inclusion $N_X[a]\subset N_X[b]$ becomes $N_X[a]= N_X[b]$
in a vertex-transitive graph. This proves that $a$ and $b$ are twin vertices
and that $D_0(X) \subset {\rm Twins}(X) $.

(ii) This follows directly from the fact
that every automorphism $\f : X \to X$ induces 
an automorphism $\f_\star : X_\star \to X_\star$
defined by $\f_\star(x_\star)=(\f(x))_\star$ for
every vertex $x$ of $X$.

(iii) Since $N_X[x]$ is a clique of $X$, $N_X[x] \subset N_X[y]$ 
for any vertex $y$ adjacent to $x$. So, by vertex transitivity,  
$N_X[x]=N_X[y]$ and $x_{\star}=N_X[x]$. Now, as $N_X[x]=N_X[y]$ 
whenever $y\sim x$, we get that $z\sim y$ 
and $y\sim x$ implies $z\sim x$ for all vertices $y$ and $z$
and this proves that the connected component 
containing $x$ is equal to $N_X[x]$.
\end{proof}

\begin{prop}\label{prop-D0+vt-imply-complete}
If $X$ is a $0$-dismantlable and vertex-transitive graph, then $X$ is a complete graph.
\end{prop}

\begin{proof}
Let $X$ be a 0-dismantlable and vertex-transitive graph. 
We prove that $X$ is a complete graph by induction on $k=\vert V(X)\vert$.
If $\vert V(X)\vert =1$, $X\cong pt=K_1$ and there is nothing to prove.
Let $k\geq 1$ and let us suppose that any 0-dismantlable 
and vertex-transitive graph with at most $k$ vertices
is a complete graph. 
Let $X$ be a 0-dismantlable and vertex-transitive graph
with $k+1$ vertices. By
Lemma \ref{lem-vt-imply-twin-egale-D0}(ii), the graph $X_\star$ is a vertex-transitive graph 
and, by Lemma \ref{lem-X-0-dismants-on-Xstar}, $X_\star \in D_0$.
As $X\in D_0$ and $\vert V(X)\vert \geq 2$,
$D_0(X)\neq \emptyset$ and, by Lemma \ref{lem-vt-imply-twin-egale-D0}(i),
${\rm Twins}(X)=D_0(X)\neq \emptyset$.
So, $\vert V(X_{\star}) \vert < \vert V(X) \vert $ and,
by induction hypothesis, $X_\star$ is a complete graph.
As ${\rm Twins}(X_\star)=\emptyset$, by Proposition \ref{sab61}(ii), 
we conclude that $X_\star \cong pt$
and this proves that $X$ is a complete graph.
\end{proof}

\noindent
Given the equivalence between $0$-dismantlability 
for graphs and strong collapsibility for clique complexes 
(case $k=0$ of Proposition \ref{prop_link}), 
Proposition \ref{prop-D0+vt-imply-complete} is nothing but 
\cite[Corollary 6.6]{barmin12} in the restricted case of flag complexes. 
But the proof given here doesn't refer to the fixed points scheme
and can be generalised by introducing the notion of
$i$-complete-transitive graphs. In what follows, 
if $(v_1,\cdots,v_k)\in V(X)^k$,
then the subgraph of $X$ induced by $(v_1,\cdots,v_k)$ refers to
$X[v_1,\cdots,v_k]$, the subgraph induced by $\{v_1,\cdots,v_k\}$.

\begin{definition}
Given $i \geq 1$, a graph $X$ will be called \emph{$i$-complete-transitive} 
if for all $1 \leq k\leq i$ and all pairs 
$$\left\{ (x_{1},\cdots,x_{k}),
(x'_{1},\cdots,x'_{k})\right\}$$ of $k$-tuples
of pairwise distinct vertices inducing a complete subgraph  of $X$,
there exists $f \in {\rm Aut}(X)$ such that 
$f(x_{j})=x'_j$ for all $j \in \{1,\cdots,k\}$. 
\end{definition}

\noindent  
The set of $i$-complete-transitive graphs contains
the set of  {\sl $i$-transitive graphs } previously introduced
in \cite{cameron, meredith}. 
We note that 1-complete-transitive graphs are just vertex-transitive
graphs and a 2-complete-transitive graph is a vertex-transitive
and arc-transitive graph.
Complete-transitive graphs are  a generalisation of
arc-transitive graphs  but to complete subgraphs and not to
paths, as are the $i$-arc-transitive graphs \cite{gr}.
Kneser graphs are examples of $i$-complete-transitive graphs 
for all integers $i$.
We now have the following generalisation of 
Proposition \ref{prop-D0+vt-imply-complete}.

\begin{prop}\label{prop-Dk+k-vt-imply-complete}
Let $X$ be a graph and $k \in \mN$.  
If $X \in D_k$ and if $X$ is $(k+1)$-complete-transitive, 
then $X$ is a complete graph.
\end{prop}

\begin{proof}
We prove it by induction on $k\geq 0$.

For $k=0$, the claimed assertion is given by Proposition \ref{prop-D0+vt-imply-complete}.

Let $k\geq 0$ and suppose that any $k$-dismantlable 
and $(k+1)$-complete-transitive graph 
is a complete graph. 
Let $X$ be a $(k+1)$-dismantlable 
and $(k+2)$-complete-transitive graph
and  $x \in D_{k+1}(X)$.
We will verify that the $k$-dismantlable graph $N_X(x)$ is a $(k+1)$-complete-transitive graph.
Let $\{(x_{1},\cdots,x_{k+1}),$ $(x'_{1},\cdots,x'_{k+1})\}$ 
be a pair of
sets of vertices of cardinal $k+1$, each of them inducing a clique 
of $N_X(x)$, the pair of sets
$\{(x,x_{1},\cdots,x_{k+1}),(x,x'_{1},\cdots,x'_{k+1})\}$ 
is of cardinal $(k+2)$, each of them inducing a clique
of $X$. By $(k+2)$-complete-transitivity of $X$, there exists 
$f \in {\rm Aut}(X)$ such that $f(x)=x$ and
$f(x_{i})=x'_i$ for all $i \in \{1,\cdots,k+1\}$.
In particular, $\varphi=f_{\vert N_X(x)}$ verifies
$\varphi \in {\rm Aut}(N_X(x))$ and 
$f(x_{i})=x'_i$ for all $i \in \{1,\cdots,k+1\}$.
So, $N_X(x)$ is a $k$-dismantlable and $(k+1)$-complete-transitive graph.
By induction hypothesis, $N_X(x)$ is a complete graph. 
As $X$ is vertex-transitive, by Lemma \ref{lem-vt-imply-twin-egale-D0}(iii),
it means that, for any vertex $x$ of $X$,
the connected component of $X$ containing $x$ is a complete subgraph.
Now, $X$ is connected since it is in $D_\infty$ 
 and so,  $X$ is a complete graph.
\end{proof}

Let us recall the notion of evasivity for simplicial 
complexes \cite{bjorner,kozlov}. 
One can present it as a game:
given a (known) simplicial complex $\cK$ with vertex 
set $V$ of cardinal $n$, through a series of questions,
a player has to determine if a given (unknown) subset 
$A$ of $V$ is a simplex of $\cK$. The only possible questions  
for the player are,
 for every vertex $x$ of $V$,  \og is $x$ in $A$\:?\fg.
The complex $\cK$ is called \textsl{non-evasive} if,
whatever is the chosen subset $A$ of $V$, the player can 
determine if $A$ is a simplex of $\cK$ 
in at most $(n-1)$ questions. 
By restriction to flag complexes,
we get the notion of non-evasiveness for graphs:

\begin{definition}
A graph $X$ is called \textsl{non-evasive} if
$\cl(X)$ is a non-evasive simplicial complex.
\end{definition}

In other terms, a graph $X$ is called \textsl{non-evasive} if
for any  $A\subset V(X)=\{x_1,\cdots,x_n\}$ one can
guess if $A$ is a complete subgraph of $X$ in
at most $n-1$ questions of the form \og is $x$ in $A$\:?\fg.
In \cite{barmin12}, the authors note that a complex $\cK$ 
is non-evasive if,  and only if, 
there is an integer $n$ such that $\cK$ is $n$-collapsible. 
The equivalence between $k$-dimantlability of graphs
and $k$-collapsibility of flag complexes
(cf. Proposition \ref{prop_link}) gives:  

\begin{prop}\label{prop-NE-egale-Dinfini}
A graph $X$ is non-evasive if, and only if, $X$ is in $D_{\infty}$. 
\end{prop}

\noindent 
The Evasiveness Conjecture for simplicial complexes 
states that \emph{every non-evasive
vertex homogeneous simplicial complex is a simplex} 
\cite{kss84}. Again, its
 restriction to clique complexes 
can be formulated in terms of graphs:
%; see also \cite{bjorner},\cite{kozlov} 

\begin{conj}[Evasiveness conjecture for graphs]\label{conj1}
Let $X$ be a graph,  if $X$ is in $D_{\infty}$ 
and vertex-transitive, 
then $X$ is a complete graph. 
\end{conj}
This formulation should not be confused with the evasiveness
 conjecture for monotone graph properties 
 \cite{bjorner,kss84, kozlov}. Let's note that
 Proposition \ref{prop-Dk+k-vt-imply-complete} 
 is a direct consequence of the
 conjecture, if that one is true.
Following a remark due to Lov\'asz, Rivest and 
Vuillemin \cite{rivest} 
pointed out that a positive answer 
to the evasiveness conjecture implies 
that a finite vertex-transitive graph with a clique  
 which intersects all its maximal cliques is a complete graph. 
Actually,  they prove that a graph with a clique  
 which intersects all its maximal cliques is
non-evasive, i.e. is in $D_{\infty}$ by
Proposition \ref{prop-NE-egale-Dinfini}. 
Remark 3.3 of \cite{bar13} is another formulation 
of this result. Indeed, the 1-skeletons of star clusters 
are exactly the graphs  which contain a clique 
intersecting all maximal 
cliques. Theorem \ref{theo-dismantling-and-payan} 
will give a stronger result.

We recall
that if  $Y$ and $Z$ are two subgraphs of a graph $X$,
$Y \cap Z$ will denote the subgraph $(V(Y)\cap V(Z), E(Y)\cap E(Z))$
and one says that $Y$ intersects $Z$ if $V(Y\cap Z)\neq \emptyset$.

\begin{lemma}\label{lemma_payan-properties}
If $X$ is a graph with a clique 
$A$  which intersects all maximal cliques of $X$
and $x$ is in $V(X)\setminus V(A)$, then:  
\smallbreak

(i)  $A$  intersects  all maximal cliques of $X-x$.

%(ii) If $ V(A) \subset \bigcap_{x \in V(X)\setminus V(A)}N_X(x)$, 
%then $X \rd{0} A$.

(ii) $A \cap N_X(x)$ is a  
complete graph   which intersects all maximal cliques of $N_X(x)$.
\end{lemma}

\begin{proof}

(i) Let $K$ be a maximal clique of $X-x$. 
If $K$ is a maximal clique of $X$, then, by property of $A$, 
$K \cap A \neq \emptyset$. Otherwise, $K+x$ is a 
maximal clique of $X$ and, by property of
$A$, $(K+x)\cap A\neq \emptyset$. 
Since $x\not \in A$, it implies $K \cap A \neq \emptyset$.

%\noindent (ii)  The inclusion 
%$ V(A) \subset \bigcap_{x \in V(X)\setminus V(A)}N_X(x)$
% implies that any vertex $x$ is dominated by any vertex of $A$.

\noindent (ii) If $K$ is a maximal clique in $N_X(x)$, 
then $K + x$ is a maximal clique
of $X$ and, by property of $A$, $(K + x)\cap A \neq \emptyset$. 
As $x\not\in A$, $K\cap A \neq \emptyset$ and also 
$K \cap \bigl(A \cap N_{X}(x)\bigr) \neq \emptyset$ since $K\subset N_X(x)$.
\end{proof}

\begin{theorem}\label{theo-dismantling-and-payan}
Let $X$ be a graph. If $A$ is a clique   which 
intersects all maximal cliques of $X$,
then  $X \in D_{a-2}$ 
with $a=\vert V(A)\vert \geq 1$.
Moreover, $X \rd{\scriptstyle{a-2}} A$ 
if $\vert V(A)\vert \geq 2$.

\end{theorem}

\begin{proof}
Let $X$ be a graph with $n$ vertices
and  $A$  a clique of $X$   which 
intersects all maximal cliques of $X$
with $a=\vert V(A)\vert \geq 1$.

If $a=1$, then $X$ is a cone whose apex is the
 vertex of $A$, that is $X \in D_{-1}$.

If $a=2$, let us denote by $u$ 
and $v$ the vertices of $A$.
If $x \in V(X)\setminus V(A)$ with $x \sim u$ and 
$x\not\sim v$, then $x$ is dominated by $u$.
Indeed, let $y \sim x$ and $K$ a maximal clique of $X$
containing $x$ and $y$, by property of $A=\{u,v\}$,
 either $u\in K$ or $v\in K$. As $v\in K$ contradicts 
 $x\not\sim v$, we conclude that $u\in K$
and $u \sim y$.
In conclusion, a vertex not in $V(A)$ 
is  dominated by
$u$ or  $v$ or  both together. So, $X \rd{0} A$
and $X \in D_0$.

Let us suppose that $a\geq 3$, 
we will prove that  $X \rd{\scriptstyle{a-2}} A$ by 
induction on $n =\vert V(X)\vert \geq 3$.
For $n=3$, we have $X=A$ and $X$ is a complete graph.
Now, suppose that the assertion of the theorem
is true for some  $n\geq 3$
and
let us  consider a graph $X$ with $n+1$ vertices and
a clique $A$
which intersects all its maximal cliques.
%If $\vert V(A)\vert \in \{1,2\}$, we know that $X \in D_{a-1}$. Let us suppose that $\vert V(A)\vert\geq 3$.
If $ V(A) \subset N_X(x)$ for every 
$x \in V(X)\setminus V(A)$, then every vertex of $A$
is an apex of $X$ and $X\rd{0}A$.
%If $ V(A) \not\subset \bigcap_{x \in V(X) \setminus V(A)}N_X(x)$,
% there is some vertex $x$ 
If  $ V(A) \not \subset N_X(x)$ for some 
$x \in V(X)\setminus V(A)$,
then  $\vert V(A) \cap N_X(x) \vert \leq a-1$. 
By Lemma \ref{lemma_payan-properties}(ii), 
$A \cap N_X(x)$ is a  
complete  subgraph  of $N_X(x)$ which intersects 
all its maximal cliques 
and, by induction hypothesis applied to $N_X(x)$, we get 
$N_X(x)\rd{a-3} A \cap N_X(x) \rd 0 pt$
as $\vert A \cap N_X(x) \vert -2\leq (a-1) -2=a-3$.
This proves that $x \in D_{a-2}(X)$, 
that is   $X \rd{a-2} X-x$.
Moreover,  by Lemma \ref{lemma_payan-properties}(i),
the induction
hypothesis implies that  $X-x \rd{\scriptstyle{a-2}} A$.
The composition $X \rd{a-2} X-x \rd{\scriptstyle{a-2}} A$ 
proves that $X \rd{\scriptstyle{a-2}} A$.

Of course, we conclude that $X \in D_{a-2}$ because $A \rd 0 pt$.
\end{proof}

\section{Appendix: the Bing's house}

\begin{figure}[h]
\begin{center}
\includegraphics[width=14cm]{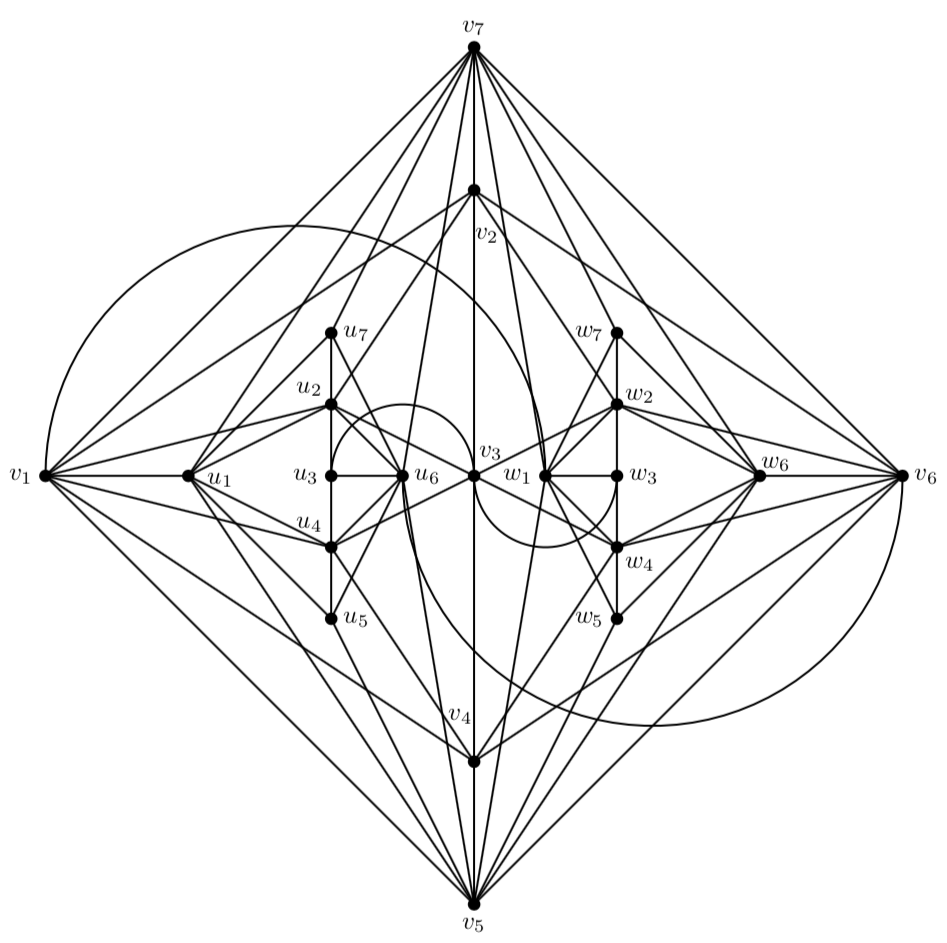}
\caption{The graph $BH$, the 1-skeleton of a 
triangulation of the Bing's House, given in \cite{cyy}.}
\label{BH-figure}
\end{center}
\end{figure}

The graph $BH$ (Fig. \ref{BH-figure}), given in \cite{cyy} where it is denoted by $G_b$,
is the 1-skeleton of 
%The clique complex $\cl(BH)$ is a triangulation of 
the topological Bing's House, a space which is known 
to be contractible but non collapsible. 
In \cite{cyy}, the
 authors give an explicit sequence of deformations
of $BH$, by using additions and deletions of edges, 
in order to prove that the Bing's House is deformable 
to the simplicial complex reduced to a point by
a sequence of expansions or reductions (proving 
that the Bing's House has the simple homotopy 
type of a point). We give here a more precise result
with the proof 
(Proposition \ref{BH_DH} for the Bing's House) 
of the existence of a graph $W_{BH}$ such that:
$$BH \ep 0 W_{BH} \rd 1 pt   $$ 
Let be $G_0=BH$, we do the following transformations
(illustrated in Fig. \ref{from-BH-to-G4}):

\begin{enumerate}
\item Within $N_{G_0}(u_1)$: $0$-additions of 
vertices $u'_{1}$ and $u''_{1}$ linked 
to $u_1,v_1,u_2,u_7,v_7$ and $u_1,v_1,u_4,u_5,v_5$
respectively. 
Note that $u'_{1} \dom u_1$ and $u''_{1} \dom u_1$. 
After that, since $N_{G_0+u'_{1}+u''_{1}}(u_1)$ is 
made of two $4$-wheels glued in vertex $v_1$, and 
thus is a $0$-dismantlable graph, it is possible 
to $1$-delete $u_1$. Let us note $G_1=G_0+u'_{1}+u''_{1}-u_1$.
\item Within $N_{G_1}(v_1)$: $0$-additions of 
vertices $v'_{1}$ and $v''_{1}$, 
linked to $v_1,u'_{1},u_2,v_2,v_7$  and $v_1,u''_{1},u_4,v_4,v_5$
respectively, 
and $1$-deletion of vertex $v_1$. Let us 
note $G_2=G_1+v'_1+v''_1-v_1$.
\item Within $N_{G_2}(w_6)$: $0$-additions of 
vertices $w'_{6}$ and $w''_{6}$,  
linked to $w_6,v_{6},v_5,w_5,w_4$ and $w_6,v_{6},w_2,w_7,v_7$
respectively, 
and $1$-deletion of vertex $w_6$. Let us 
note $G_3=G_2+w'_6+w''_6-w_6$.
\item Within $N_{G_3}(v_6)$: $0$-additions of 
vertices $v'_{6}$ and $v''_{6}$, 
linked to $v_6,w'_{6},v_5,v_4,w_4$ and $v_6,w''_{6},w_2,v_2,v_7$
respectively, 
and $1$-deletion of vertex $v_6$. Let us note $G_4=G_3+v'_6+v''_6-v_6$.
\smallbreak
 
Now, observe that $G_4 \in D_1$. 
Indeed (see Fig. \ref{G4-is-in-D1}):
\smallbreak
\item Vertices $v_5$ and $v_7$ are 
in $D_1(G_4)$ since their neighbourhoods 
are paths $w_1w_5w'_6v'_6v_4v''_1u''_1u_5u_6$ 
and $u_6u_7u'_1v'_1v_2v''_6w''_6w_7w_1$ respectively. 
Let us note $G_5=G_4-v_5-v_7$.
\item  The graph $G_5$ is a planar triangulated 
graph  which is $0$-dismantlable with  $22$ 
successive $0$-deletions.
\end{enumerate}
\noindent The switching property $(\dag)$ finishes the proof
and 
$W_{BH}=BH+u'_1+u''_1+v'_1+v''_1+w'_6+w''_6+v'_6+v''_6$.

\begin{figure}[h]
\begin{center}
\includegraphics[width=12cm]{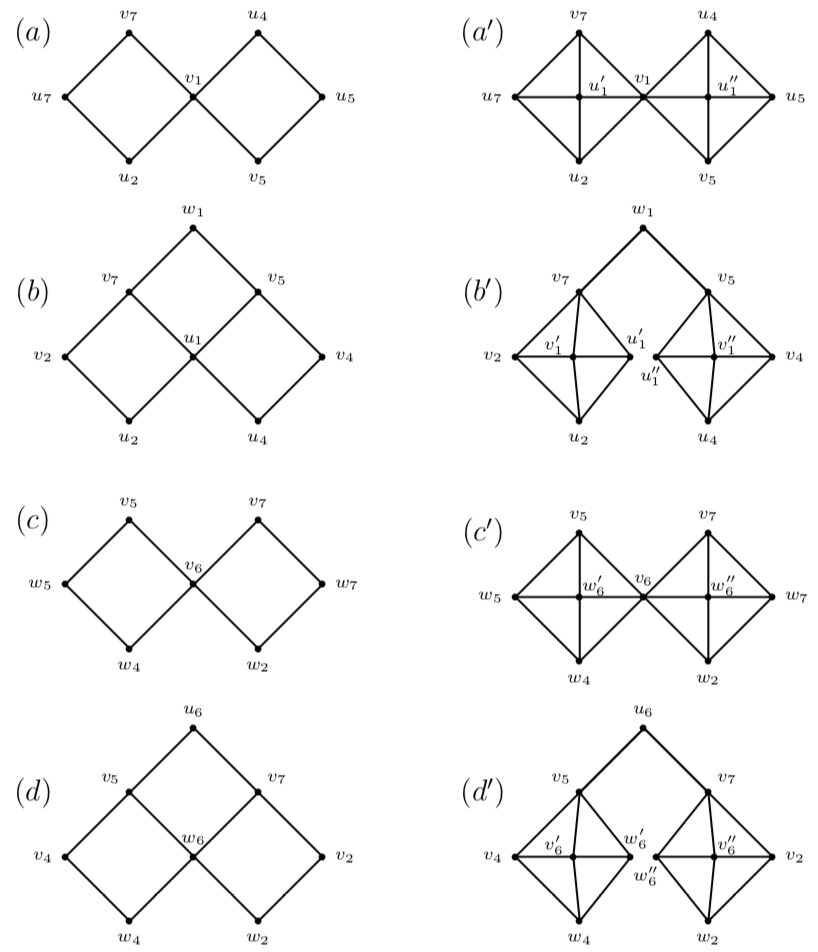}
\caption{
Eight 0-additions and four 1-deletions of vertices.
(left) Neighbourhoods in $BH=G_0$ of vertices  $u_1$ $(a)$, $v_1$ $(b)$, 
$w_6$ $(c)$ and $v_6$ $(d)$. 
(right) 
$(a') $ neighbourhood of vertex $u_1$ in $G_0+u'_{1}+u''_{1}$,
$(b') $ neighbourhood of vertex $v_1$ in $G_1+v'_{1}+v''_{1}$
with $G_1=G_0+u'_{1}+u''_{1}-u_1$,
$(c') $ neighbourhood of vertex $w_6$ in $G_2+w'_{6}+w''_{6}$
with $G_2=G_1+v'_{1}+v''_{1}-v_1$,
$(d') $ neighbourhood of vertex $v_6$ in 
$G_3+v'_{6}+v''_{6}$
with $G_3=G_2+w'_{6}+w''_{6}-w_6$
}
\label{from-BH-to-G4}
\end{center}
\end{figure}

%\psset{xunit=0.75cm,yunit=0.8cm,dotsize=5pt 0}
\begin{figure}[h]
\begin{center}
\includegraphics[width=15cm]{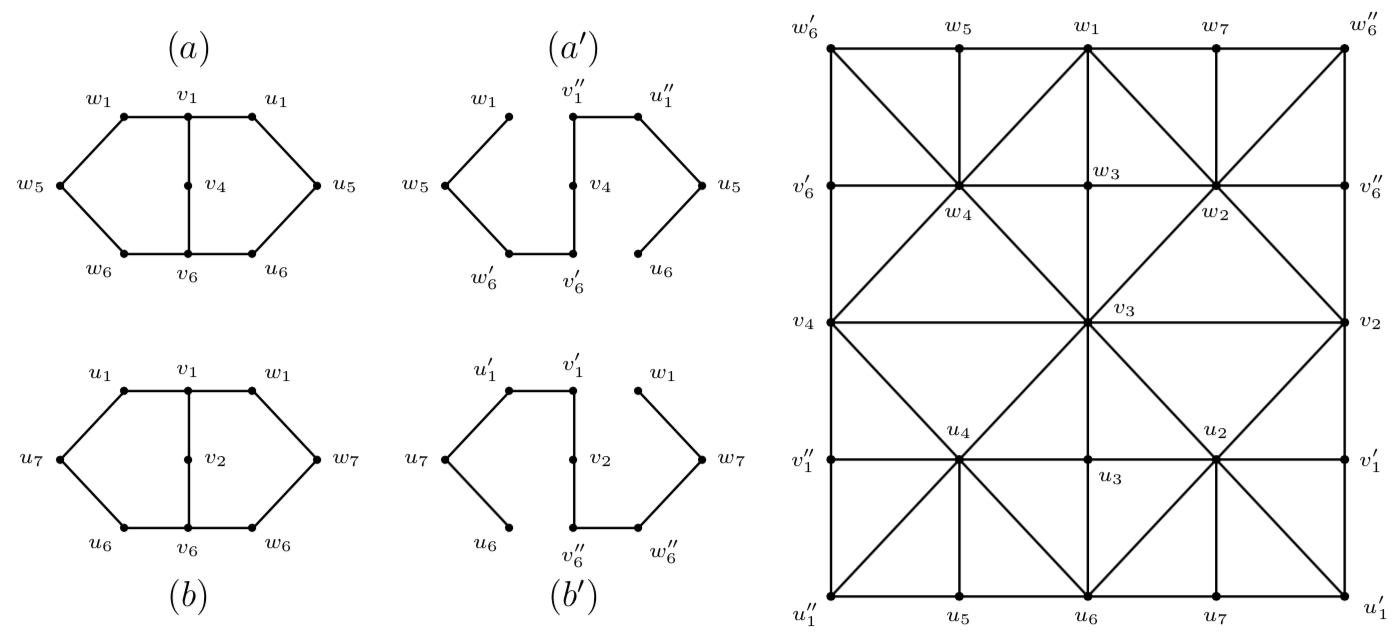}
\caption{
The graph $G_4$ obtained from $BH$ by 0-addition of vertices $u'_1$, $u''_1$, 
$v'_1$, $v''_1$, $w'_6$, $w''_6$, $v'_6$, $v''_6$ and 
1-deletion of vertices $u_1$, $v_1$, $w_6$, $v_6$
is 1-dismantlable.  
(left) Neighbourhoods in $BH$ of vertices  $v_5$ $(a)$ and $v_7$ $(b)$, 
(center) 
Neighbourhoods in $G_4$ of vertices  $v_5$ $(a')$ and $v_7$ $(b')$
(right) The graph $G_4-v_5-v_7$ is 0-dismantlable.
}
\label{G4-is-in-D1}
\end{center}
\end{figure}

\section{Acknowledgements}
We particularly thank one of the reviewer for 
his/her careful reading and extensive comments, 
which have helped to improve the paper, in particular the second part of section 4.1. 
The reviewer also suggested remark 16 and some improvements in the proof of theorem \ref{theo-dismantling-and-payan}.

\end{document}